\documentclass[12pt,letterpaper]{article}
\usepackage{amssymb}
\usepackage{amsmath}
\usepackage{amsthm}
\usepackage{graphicx}
\usepackage{verbatim}
\usepackage{fullpage}
\usepackage{yfonts}
\usepackage[T1]{fontenc}
\usepackage{appendix}
\usepackage{hyperref}

\def\reals{\mathbb{R}}

\newtheorem{definition}{Definition}
\newtheorem{theorem}{Theorem}
\newtheorem{proposition}{Proposition}
\newtheorem{question}{Question}

\newtheorem{corollary}{Corollary}

\title{Why is Helfenstein's claim about equichordal points false?}
\author{Marek Rychlik\\
  University of Arizona\\
  Department of Mathematics, 617 N Santa Rita Rd, P.O. Box 210089\\
  Tucson, AZ 85721-0089, USA\\}

\date{\today}

\begin{document}
\maketitle
\begin{abstract}
  This article explains why a paper by Heinz G. Helfenstein entitled
  ``Ovals with equichordal points'', published in J.~London
  Math.~Soc.~31, 54--57, 1956, is incorrect. We point out a
  computational error which renders his conclusions invalid. More
  importantly, we explain that the method cannot be used to solve the
  equichordal point problem with the method presented there.  Today,
  there is a solution to the problem: Marek R. Rychlik, ``A complete
  solution to the equichordal point problem of Fujiwara, Blaschke,
  Rothe and Weizenb\"ock'', Inventiones Mathematicae 129 (1),
  141--212, 1997. However, some mathematicians still point to
  Helfenstein's paper as a plausible path to a simpler solution. We
  show that Helfenstein's method cannot be salvaged. The fact that
  Helfenstein's argument is not correct was known to Wirsing, but he
  did not explicitly point out the error. This article points out the
  error and the reasons for the failure of Helfenstein's approach in
  an accessible, and hopefully enjoyable way.
\end{abstract}
\section{The Equichordal Point Problem}
The equichordal point problem enjoyed significant popularity since its
original formulation by Fujiwara in 1916 and Blaschke, R\"othe and
Weizenb\"ock in 1917, because it can be formulated in easy to
understand terms of elementary geometry, and it is hard to solve.  The
starting point is the definition of an equichordal point:
\begin{definition}
  Let \(C\) be a Jordan curve and let \(O\) be a
  point inside it. This point is called \emph{equichordal}
  if every chord of \(C\) through this point has the same length.
\end{definition}
Then the equichordal point problem may be formulated
as follows:
\begin{question}
  Is there a curve with two equichordal points?
\end{question}
Why two? Because circles and a lot of other shapes have one equichordal
point, and because Fujiwara pointed out that it is impossible for a
shape to have three equichordal points.

The full solution to the problem appears in the article
\cite{Rychlik97}. The paper is considered (and it is!) rather hard to
read and its length is 72 pages. Thus, although it is a great resource
for anyone studying this and related problems, it is not always easy
to extract the information one needs.

In this set of notes we use the information provided in
\cite{Rychlik97} to construct a counterexample to a published article
by Helfenstein \cite{Helfenstein56} in 1956.  Helfenstein made a claim
which would lead to a simple solution of the equichordal point problem
(under 10 pages, perhaps) if it is augmented with a few relatively
easy facts to prove.  It has been hope of many that such a simple
proof exists. However, as we will see, Helfenstein's paper is
incorrect, and thus there is no hope for a simple proof, at least
along the lines of Helfenstein's argument.

The convex geometry community, in which the equichordal point problem,
and our solution of it, are quite well known.  The community has had
an especially hard time coming to grips with the hard analytical
methods used in our article (and also prior articles of Wirsing
\cite{Wirsing58} and Sh\"afke and Volkmer \cite{ShafkeVolkmer92}).  We
found on several occasions that the argument of Helfenstein continues
to have some legitimacy because no one has explicitly shown where it
is incorrect \cite{KleeRevision60}.  At the end of this paper we cite
Gr\"unbaums's argument from \cite{KleeRevision60} which is
representative of this opinion, although the experts on the problem
(including Wirsing) have clearly dismissed Helfenstein's
paper. Therefore, it will be beneficial to analyze Helfenstein's
argument from today's perspective, and explain why it is incorrect. We
hope that the reading is entertaining and allows one to understand
some of the trappings of the problem, and perhaps even appreciate the
length and complexity of our solution.

Helfenstein's paper contains an incorrect statement which must have
resulted from an error in a mundane calculation, involving only
elementary calculus. This will be clear from what follows. With the
aid of a Computer Algebra System (CAS), we reconstructed and corrected
the intermediate calculations, and arrived at the opposite conclusion,
which clearly shows an error in Helfenstein's argument in an
elementary way.

More importantly, the main idea of Helfenstein's paper is also
incorrect, and it cannot be salvaged by simply correcting the error in
calculation he made. We show this in the strongest possible way: we
construct a counterexample by referring to the relevant portions
of \cite{Rychlik97}. However, we will make the argument as simple
and as self-contained as possible.
\section{A summary of Helfenstein's paper}
In 1954 Heinz Helfenstein submitted an article \cite{Helfenstein56},
in which he claims that there is no oval with 2 equichordal points
that is 6 times differentiable. He calls a curve with 2 equichordal
points a \emph{2e-curve}. We will use this abbreviated term below, as
synonymous with ``curve with 2 equichordal points''.

We proceed to summarize the terminology, technique and results of
Helfenstein's paper. For simplicity, we will assume that the curve
\(C\) under consideration is a convex oval and it is symmetric in
various ways (cf. Figure~\ref{figure:points_and_lines}):
\begin{enumerate}
\item
  It is symmetric with respect to a point \(O\)
  inside \(C\).
\item
  Let \(R\) and \(S\) be the two hypothetical equichordal
  points. We may assume that \(O\) bisects the interval \(RS\).
\item 
  \(C\) is symmetric with respect to the straight line passing
  through \(R\) and \(S\).
\item 
  Let \(A\) and \(B\) be the two points collinear with \(R\) and
  \(S\) which belong to \(C\).  We may assume that the distance
  \(|AB|=1\).  Moreover, due to prior symmetries, \(O\) bisects
  \(AB\).
\item 
  The distance \(BS\) is called \(c\). By the symmetries
  assumed, we have \( AR=c \).
\item
  The construction is valid when \( 0 < c < 1 \).
  The order of the points \(A\), \(R\), \(O\), \(S\) and
  \(B\) on the straight line on which all these points lie is:
  \begin{enumerate}
  \item \( A, R, O, S, B\) if \(0 < c < 1/2\);
  \item \( A, S, O, R, B \) if \(1/2 < c < 1\).
  \end{enumerate}
  For \(c=1/2\), \(R=O=S\).
\end{enumerate}

\begin{figure}
  \begin{center}
    \includegraphics[width=.8\textwidth]{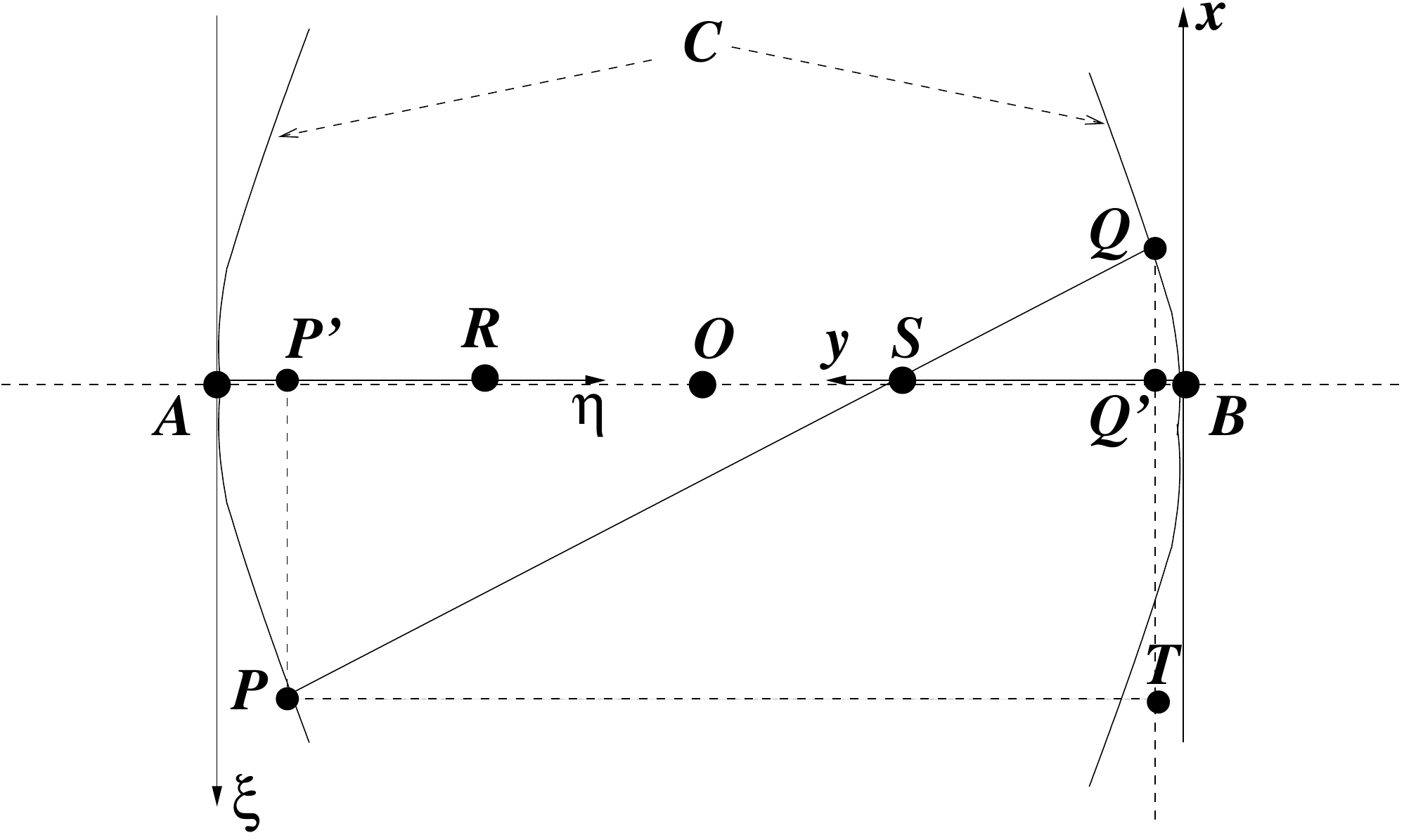}
  \end{center}
  \caption{\label{figure:points_and_lines} Illustration of Helfenstein's notation.}
\end{figure}
It should be stated that the above picture of a hypothetical
2e-curve is correct, based on many independent analyses. In
particular, the symmetries are well established.

The next assumption used in the paper \cite{Helfenstein56} is that
locally near \(A\) and \(B\) the curve \(C\) may be represented by a
graph of a function. Helfenstein uses two orthogonal coordinate
systems, one centered at \(A\) and one centered at \(B\). The
coordinates of the system centered at \(A\) are called \(\xi\) and
\(\eta\), and the positive direction of the \(\eta\) axis is
\(AB\). The coordinates of the system centered at \(B\) are called
\(x\) and \(y\), and the positive direction of the \(y\) axis is
\(BA\), so that the \(y\) and \(\eta\) axes point in the opposite
directions.  We assume that \(C\) near the points \(B\) and \(A\) may
be represented by the equations: \( y = f(x) \) and \( \eta = f(\xi)
\), respectively.  There is an agreement of results supporting the
claim that \(C\) is represented by a smooth function \(f(x)\) near the
points \(A\) and \(B\). It should be emphasized that \(f(x)\) is
\emph{locally defined}, i.e.  its domain is some interval
\((-\epsilon,\epsilon)\), where \(\epsilon>0\). There is a proof that
\(\epsilon\) may be as large as \(1/2\), but this will not be material
in these notes. One can also prove that \(f(x)\) is
\emph{real-analytic}, i.e. it may be represented by a power series
convergent on the interval \(|x|<\epsilon\). Again, the analyticity is
not material in these notes, but Helfenstein assumes that the function
has six derivatives. It should be noted that Helfenstein's assumption
in regard to differentiablility is faulty. The fact that he assumes
six-fold differentiability is a result of a computational error, as
will be demonstrated below.

The next important construction in Helfenstein's paper is that of a
functional equation satisfied by \(f(x)\). The derivation presented in
the Helfenstein's paper is correct, and is consistent with the
construction used in our solution of the equichordal point problem
\cite{Rychlik97}.  It nicely illustrates the transition from geometry
to analysis, which is a hallmark feature of the equichordal point
problem. We will repeat the Helfenstein's construction here.

Let \(P(\xi,\eta)\) be a point near \(A\) and let \(Q(x,y)\) be a
point near \(B\), both on the curve \(C\) and both represented in
the respective coordinate systems. Let \(P'\) and \(Q'\) be the
orthogonal projections of \(P\) and \(Q\) onto the line
\(AB\). Let \(T\) be the projection of \(P\) onto the line
\(QQ'\), perpendicular to \(AB\). Helfenstein observes that the
triangles \(QQ'S\) and \(QTP\) are similar. From this observation,
the following equations result:
\begin{eqnarray*}
\frac{x}{\sqrt{x^2+(c-y)^2}} &=& \frac{x+\xi}{1},\\
\frac{c-y}{\sqrt{x^2+(c-y)^2}} &=& \frac{1-y-\eta}{1}.
\end{eqnarray*}
By solving with respect to \(\xi\) and \(\eta\) we
obtain:
\begin{eqnarray}
\label{equation:helfenstein-xi}
\xi  &=& \frac{x}{\sqrt{x^2+(c-y)^2}} - x\\
\label{equation:helfenstein-eta}
\eta &=& \frac{-(c-y)}{\sqrt{x^2+(c-y)^2}} + (1-y).
\end{eqnarray}
By construction, \(y=f(x)\) and \(\eta = f(\xi)\).
Therefore, we obtain this functional equation:
\begin{equation}
\label{equation:functional}
f\left(\frac{x}{\sqrt{x^2+(c-f(x))^2}} - x\right)
=\frac{-(c-f(x))}{\sqrt{x^2+(c-f(x))^2}} + (1-f(x)).
\end{equation}
The manner in which Helfenstein uses this equation is also
well established: we repeatedly differentiate both sides
at \(x=0\) to obtain the consecutive derivatives
of \(f(x)\) at \(x=0\). Thus, we try to solve the equation
by finding the \(f(x)\)'s Taylor series at \(x=0\).
Helfenstein uses the following
notation, defining coefficients of the Taylor expansion at \(x=0\):
for \(n=0,1,2,\ldots\):
\begin{equation}
  \label{equation:coefficients}
  \frac{d^nf(x)}{dx}\bigg|_{x=0}=n!\,a_n.
\end{equation}
Often we will write \(a_n(c)\) instead of \(a_n\) when it is necessary
to consider the dependence of \(a_n\) upon the parameter \(c\).

The consensus of several methods is that these equations can be used
to determine \(a_n\) by a recurrence relation, and thus determine the
Taylor expansion of \(f(x)\) at \(x=0\) up to an arbitrary
order. Moreover, the symmetries imply that \(f(x)\) is an even
function:
\[f(-x)=f(x).\]
This implies that \(a_{n}=0\) for odd \(n\).
Moreover, \(f(0)=0\) follows from the assumptions made,
that \(C\) passes through \(A\) and \(B\).

We come to a point where Helfenstein makes a calculation error in
calculating the third non-trivial coefficient, \(a_6\).  Only simple
calculus (chain rule) is involved. Calculating \(a_2\) and \(a_4\) by
hand would test anyone's patience, but today is conveniently done with
the aid of a Computer Algebra System (CAS).  Helfenstein calculated
\(a_2\) and \(a_4\) correctly.  Calculation of \(a_6\) must have been
very challenging without a CAS, and indeed it resulted in an important
error which affects the entire argument.

We wrote a simple CAS program which determines the coefficients \(a_n\)
for \(n\) up to \(10\).  In theory, the program can find \(a_n\) for
arbitrarily large \(n\), but even CAS consumes an amount of time that
probably grows exponentially with \(n\).  The CAS we used is the open
source, free system Maxima \cite{Maxima}, although any other CAS can
solve this problem. We included our program as Appendix~A.

The results are presented below for even \(n\) only. Moreover, we list
numbers
\begin{equation}
  b_n = a_n\left(\frac{1}{2}+\frac{\sqrt{z}}{2}\right),
\end{equation}
which are analytic in \(z\) iff \(a_n\) are invariant under the
substitution \(c\mapsto 1-c\). Thus, it is much easier to read off the
invariance by looking at \(b_n\).

The program generated the coefficients in standard \TeX\ format.
Additionally, the program factored \(a_n\) as rational functions of \(c\),
for easy comparison of \(a_2\) and \(a_4\) with Helfenstein's paper.
Here is the result:
\begin{eqnarray*}
a_{2}&=&{{1}\over{2\,\left(2\,c^2-2\,c+1\right)}}\\
a_{4}&=&-{{12\,c^4-24\,c^3+12\,c^2-1}\over{8\,\left(2\,c^2-2\,c+1
 \right)^2\,\left(2\,c^4-4\,c^3+6\,c^2-4\,c+1\right)}}\\
a_{6}&=&{{80\,c^{10}-400\,c^9+680\,c^8-320\,c^7-500\,c^6+940\,c^5-712
 \,c^4+284\,c^3-52\,c^2+1}\over{16\,\left(2\,c^2-2\,c+1\right)^4\,
 \left(c^4-2\,c^3+5\,c^2-4\,c+1\right)\,\left(2\,c^4-4\,c^3+6\,c^2-4
 \,c+1\right)}}
\end{eqnarray*}
and
\begin{eqnarray*}
b_{2}&=&{{1}\over{z+1}}\\
b_{4}&=&-{{3\,z^2-6\,z-1}\over{z^4+8\,z^3+14\,z^2+8\,z+1}}\\
b_{6}&=&{{10\,z^5-110\,z^4-20\,z^3+204\,z^2+42\,z+2}\over{z^8+24\,z^7
 +172\,z^6+488\,z^5+678\,z^4+488\,z^3+172\,z^2+24\,z+1}}\\
\end{eqnarray*}
On the web we listed the coefficients up to \(a_{10}\)
without folding or breaking them up. The corresponding \(b_{n}\)
are clearly analytic, i.e. do not contain half-integer powers
of \(z\). This means that \(a_{n}\) is invariant under the
substitution \(c\mapsto 1-c\) for \(n\) up to 10. We can push
this calculation further, up to, say, \(n=20\), always with
the same result: it is invariant under this substitution.

Helfenstein's paper contains correct expressions for \(a_2\) and
\(a_4\). However, he did not include the expression for
\(a_6\). Since he derived a false conclusion about \(a_6\), as we
will see below the only possible explanation is that he made a
computational error in the intermediate calculations.
Helfenstein's argument is founded on an unproven claim that
if a 2e-curve exists for some value \(c\) then \(a_n\) must be
invariant under the substitution \(c\mapsto 1-c\). More precisely,
if we consider \(a_n=a_n(c)\) (i.e. as a function of \(c\)) then
the condition \(a_n(c)=a_n(1-c)\) is necessary (according to
Helfenstein) for a 2e-curve to exist for a particular value of
\(c\).

Finally, we are ready to explain Helfenstein's main argument, and how
he arrived at the erroneous six-fold differentiability condition.  He
correctly noted that the expressions for \(a_2\) and \(a_4\) are
invariant under the substitution \(c\mapsto 1-c\). He then considers
\(a_6\) as a candidate for a coefficient which is not invariant under
this substitution. In contradiction with our findings, he claims that
it is not invariant under the substitution \(c\mapsto
1-c\). Helfenstein writes:
\begin{quote}
  A sixth differentiation finally yields an expression for
  \(a_6(c)\) which is not identical to \(a_6(1-c)\).
\end{quote}
The form of \(a_6\) is omitted in Helfenstein's paper and the
intermediate calculations of it are missing. He then proceeds to
determine specific values of \(c\) which solve the equation:
\[ a_6(c) = a_6(1-c) \]
He claims that the above equation is equivalent
to a certain polynomial equation of the 9-th degree:
\[ 144\,c^9-648\,c^8+1176\,c^7-1092\,c^6+168\,c^5+798\,c^4-846\,c^3+357\,c^2-59\,c+1=0. \]
Subsequently, he demonstrates that equation does not
have roots in the range
\[ \frac{2-\sqrt{3}}{4} < c < \frac{2+\sqrt{3}}{4} \]
which is known to be the region of \(c\), outside of which there is no
2e-curve, based on elementary arguments which preceded Helfenstein's
paper. Clearly, the 9-th degree polynomial and the subsequent
conclusions are a result of a calculation error.

Helfenstein's claim is that there are no 2e-curves which are six-fold
differentiable. The reason is that he needs this much
differentiability to calculate \(a_6\).  As we have shown, the
six-fold differentiability of the function \(f(x)\) at \(x=0\) is not
sufficient to disprove the existence of a 2e-curve.  Moreover, we
verified with a CAS that Helfenstein's argument fails for curves which
are ten-fold differentiable, on the basis of our calculations of
\(a_n\) and \(b_n\) for \(n\) up to 10.

Of course, now it is time to pause, and suggest that the argument
fails for all \(n\). This shows that Helfenstein's approach cannot
succeed, even if we had unlimited computing power and find as many
coefficients \(a_n\) as necessary.
\section{More differentiability does not help}
One could hope that by finding more derivatives of \(f(x)\) we
will eventually find a coefficient \(a_n(c)\) for which \(a_n(c)\)
and \(a_n(1-c)\) do not coincide.  In this section we will
show that \(a_n(c)=a_n(1-c)\) for all \(n\), given
that all derivatives up to \(n\) exist. This, of course,
demonstrates that Helfenstein's method cannot disprove
the existence of a 2e-curve.

The key point is to understand the \emph{local} existence and
uniqueness of solutions of Helfenstein's functional equation.  It
should be noted that his paper is an attempt to disprove local
existence. What he missed is the fact that the \emph{locally
  defined} solution to his functional equation does exist!
Furthermore, he missed that the local existence does not imply
that a 2-e curve exists.

The existence and uniqueness results are contained in the
\emph{Inventiones} article \cite{Rychlik97}, but we will formulate
those results here in a manner more suitable for these notes.

The graph of a solution to the Helfenstein's functional equation
gives rise to two Jordan arcs, \(C_A\) and \(C_B\), contained in a
neighborhoods of \(A\) and \(B\) respectively.
The arcs \(C_A\) and \(C_B\) are defined by the equations
\begin{eqnarray*}
\eta&=&f(\xi),\\
y&=&f(x)
\end{eqnarray*}
in the respective coordinate system.
\begin{figure}
  \begin{center}
    \includegraphics[width=.6\textwidth]{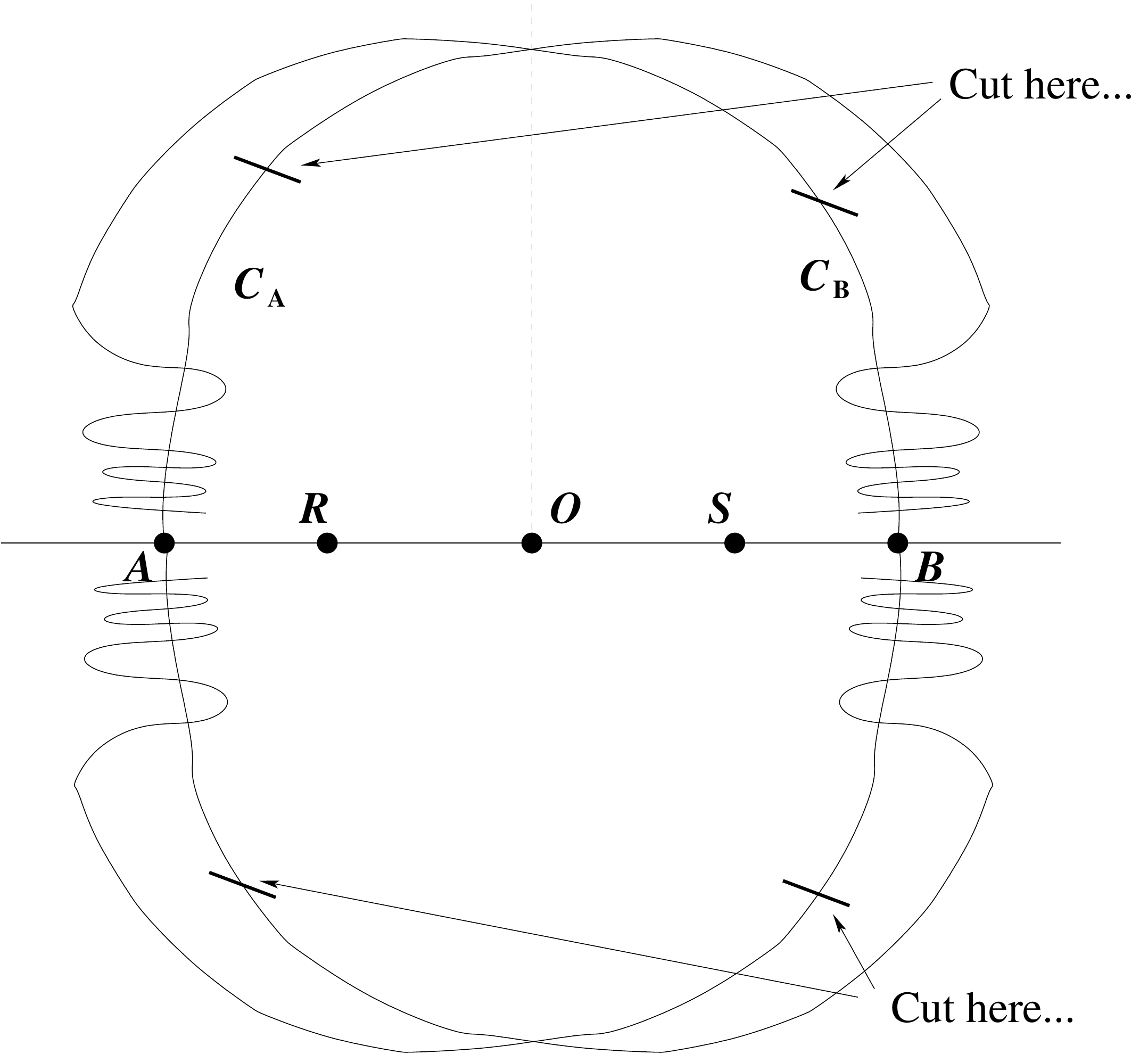}
  \end{center}
  \caption{\label{figure:homoclinic}A schematic figure of local curves connecting.}
\end{figure}
A picture is worth a thousand words. Thus, if the reader still cannot
imagine how the two arcs \(C_A\) and \(C_B\) may connect together when
they are maximally extended, without forming a 2e-curve, a plausible
scenario can be visualized by a schematic drawing in
Figure~\ref{figure:homoclinic}. The alternative would be for the two
arcs to meet exactly and form one smooth curve. The fact that they do
not meet in this manner is the thrust of the \emph{Inventiones}
article \cite{Rychlik97}. It should be noted that the figure has
perfect reflectional symmetry, both with respect to reflections in the
line \(AB\) and in the center point \(O\).  Because Helfenstein did
not look away from the line \(AB\) (i.e. outside the "cuts", which
mark the ends of Jordan arcs \(C_A\) and \(C_B\)) he failed to notice
that his assumptions may be satisfied by a {\em local equichordal
configuration}. And indeed, this is what really happens.  While the
above figure is only a schematic, we and others performed numerical
computations which are in perfect agreement with the above picture in
regard to its topology. It should be noted that the oscillations by
the curves near \(A\) and \(B\) continue up to the line \(AB\). The
size of the oscillations settles down to a fixed, positive amplitude.

The extremely important idea in understanding the equichordal point
problem has been that the problem should be phrased as a problem bout
iterations of mappings, i.e. should be framed as a problem of {\em
  dynamical systems theory}. To remain faithful to Helfenstein's
notation, we define a map \(G_c:U_c \to \reals^2\), where \(U_c \subset
\reals^3\) is an open, punctured unit disk centered at \(S(0,c)\):
\[ U_c = \left\{(x,y)\in\reals^2\,:\, 0<x^2+(y-c)^2<1 \right\} \]
The map \(G_c\) is defined by the formula \(G_c(x,y) = (\xi,\eta) \) where
\(\xi\) and \(\eta\) are given by
equations~\ref{equation:helfenstein-xi}-\ref{equation:helfenstein-eta}.
For better understanding, one should consider \(c\) a parameter, and
think of the mapping \(G:U\to\reals^2\) defined by
\(G(\xi,\eta,c)=G_c(x,y)\), where \(U\subset\reals^3\):
\[ U = \left\{(x,y,c)\in\reals^3\,:\, 0<c<1,\;0<x^2+(y-c)^2<1 \right\} \]
Occasionally, there is a technical advantage to including \(c\) in the
set of variables, for instance, when stating joint continuity,
differentiability, etc. which includes the parameter.

Geometrically, \(G_c\) acts on a point \(Q(x,y)\) in an an almost
obvious way. The preliminary idea is to map it to the point \(Q(x,y)\) to
the point \(P(\xi,\eta)\), where \(\xi\) and \(\eta\) are computed
from
equations~\ref{equation:helfenstein-xi}-\ref{equation:helfenstein-eta}. However,
this point is subsequently identified with a point
\(Q_1(x',y')\) where {\em numerically} \(x'=\xi\) and
\(y'=\eta\). Thus, the action of \(G_c\) on a point \(Q\) is described
as a composition of two maps (``walks''):
\[ Q\mapsto P\mapsto Q_1 \]
where the two walks may be described as fallows:
\begin{enumerate}
\item We start at \(Q\), and walk towards \(S\) and pass it, until we
  have walked a total distance of 1;
\item We walk from \(P\) towards \(O\) and pass it, until we reach
  \(Q_1\), and satisfy the distance condition \(|QO| = |OQ_1|\).
\end{enumerate}
In short:
\[ G_c(Q) = Q_1\]
For every \(c\in(0,1)\) the domain of \(G_c\) is the open disk about
\(S\) of radius 1, without its center (a punctured disk). The reader
should note that according to our conventions \(A\) and \(S\) like on
the same side of \(O\) when \(c<1/2\) and on the opposite sides when
\(c>1/2\). Although it is possible to enlarge the domain with
some additional conventions, we will not do so, and adhere to
the ``natural'' domain described above.

The role of the substitution \(c\mapsto 1-c\) is explained in the
following
\begin{proposition}
  \label{proposition:inverse}
  For an arbitrary \(c\in(0,1)\), Let \(Q\) be in the domain of
  \(G_c\) and le \(Q_1=G_c(Q)\). Then \(G_{1-c}\) is well defined at
  \(Q_1\) and
  \[G_{1-c}(Q_1)=Q.\] 
  In particular, the mappings \(G_c\) and \(G_{1-c}\) are inverses of
  each other and \(G_c(U_{c})=U_{1-c}\). In addition the mapping
  \(G_c:U_c\to U_{1-c}\) is a diffeomorphism.
\end{proposition}
\begin{proof}
  Let us consider point \(P(\xi,\eta)\) and another point,
  \(P_1(\tilde\xi,\tilde\eta)\) defined as the reflection of
  \(Q(x,y)\) through the point \(O\). We claim that
  \begin{enumerate}
  \item \(P_1\), \(R\) and \(Q_1\) are collinear;
  \item \(|P_1Q_1|=1\).
  \end{enumerate}
  Then the equation \(G_{1-c}(Q_1)=Q\) follows from the above claims and the definition
  of \(G\) in terms of walks. Indeed, the claims imply that the two walks
  defining \(G_{1-c}(Q_1)\) are:
  \[Q_1\mapsto P_1\mapsto Q.\]
  Both properties follow from the observation that the quadrilateral
  \(QPP_1Q_1\) is a parallelogram, Indeed, its two
  diagonals are bisected by \(O\). In particular, the side
  \(P_1Q_1\) is parallel to \(PQ\) which has length
  1. Therefore, both sides have length 1. Point \(S\) lies on the side
  \(PQ\) by definition. Thus, \(R\) lies on the opposite side
  \(P_1Q_1\) because \(O\) bisects \(RS\) by definition.
\end{proof}
This above statement and proof carefully avoid complicated algebraic
equations. An attempt to prove the above proposition by brute force
calculations is likely to fail. For a reader wanting to understand an
algebraic approach to the above lemma, we included a CAS program
Appendix~C which arithmetically verifies the claims in the above
proof.

The significance of Proposition~\ref{proposition:inverse} in the
context of Helfenstein's paper: the invariant sets of \(G_c\) and
\(G_{c}^{-1}=G_{1-c}\) are identical. It should also be noted that the
following two conditions are equivalent for a fixed \(c\):
\begin{enumerate}
\item A function \(y=f(x)\) satisfies the functional equation~\ref{equation:functional};
\item The graph \(C_B=\{(x,f_c(x))\}\) is an invariant set of \(G_c\).
\end{enumerate}
The invariance should be understood locally in the neighborhood of
\(B\) or near \(x=0\). Local invariance of a Jordan arc \(C_B\) means
that \(C_B\) is contained in the domain of \(G_c\). there is a
neighborhood \(K\) of \(B\) such that \(C_B\cap K = G_c(C_B)\cap K\).

The uniqueness of the solutions easily implies that Helfenstein's
method cannot work. Let us denote by \(f_c(x)\) any solution of the
functional equation~\ref{equation:functional}, defined in some
neighborhood of \(x=0\). If we know that the solution is unique then
\(y=f_c(x)\) is a solution of the functional equation for both \(c\)
and \(c'=1-c\). Uniqueness implies \(f_{c}(x) = f_{1-c}(x)\)
and equality \(a_n(c)=a_n(1-c)\) follows for all \(n\).
\section{Local existence and uniqueness}
It turns out that a properly formulated existence and uniqueness
theorem eliminates Helfenstein's approach as viable, but it also
eliminates the possibility that a continuous 2e-curve exists. The key
to such a strong result is a consideration of solutions to the
functional equation~\ref{equation:functional} which do not satisfy
\(f(0)=0\), but instead \(f(0)=y_0\), where \(y_0\) varies in the
range \(|y_0|<\min(c,1-c)\). Such a family provides a good coordinate
system near \(B\). Figure~\ref{figure:family} schematically depicts
the situation.  In its caption, the figure states several claims which
will be shown to reflect what has actually been proved about
the functional equation and its solutions.
\begin{figure}
  \begin{center}
    \includegraphics[width=.7\textwidth]{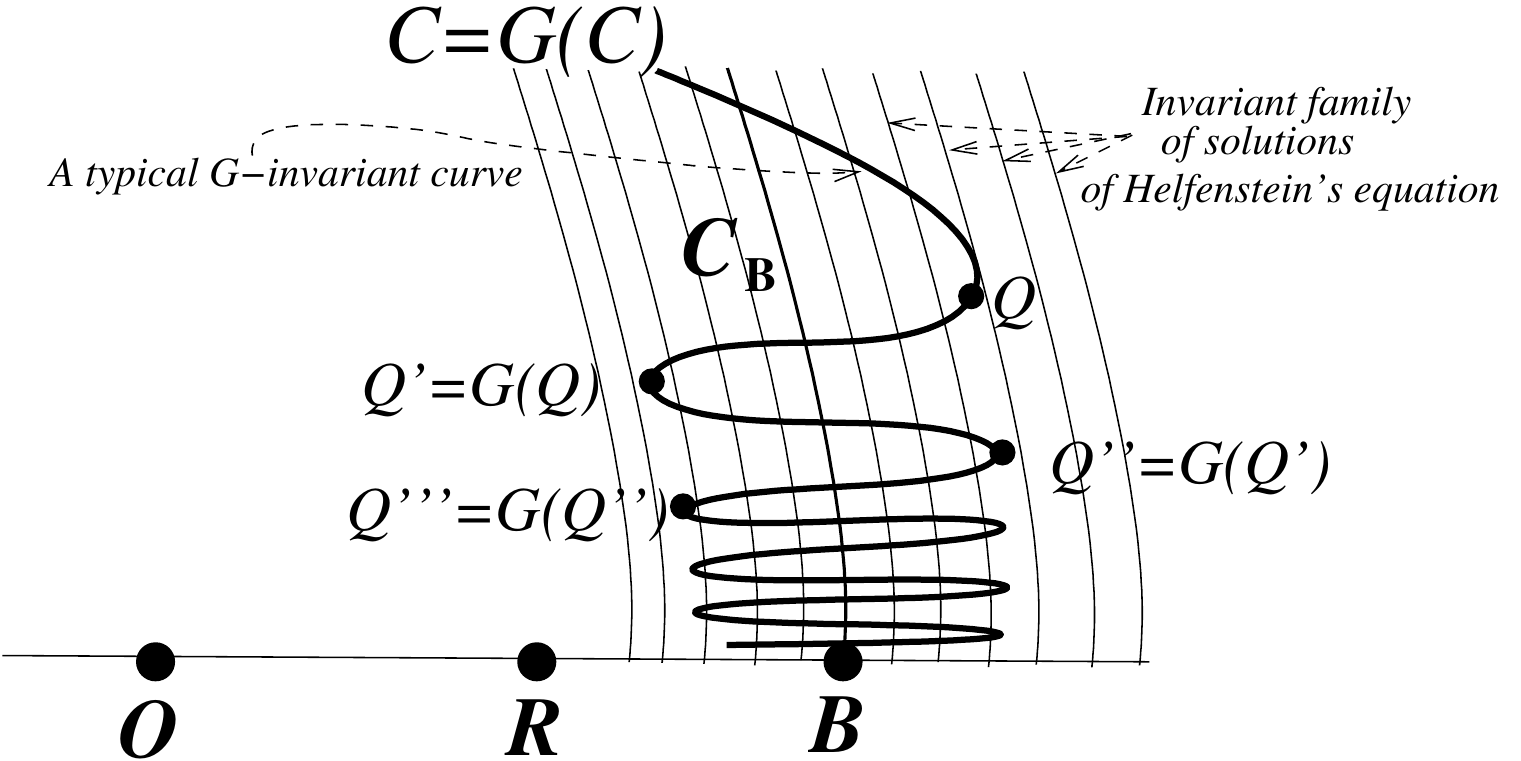}
  \end{center}
  \caption{\label{figure:family} The invariant family of curves
    \(y=F(x,y_0,c)\) near \(B\) for the map \(G=G_c\). We assume that
    \(1/2<c<1/2\), which implies that \(R\) is between \(O\) and
    \(B\).
    Only a half of each curve is drawn in the upper halfplane above
    the line \(SB\). A typical invariant curve \(C\) such that
    \(C=G_c(C)\) is depicted, as it enters the neighborhood of \(B\)
    foliated by curves of the family.
    A sample trajectory \(Q^{(i)}=Q,Q',Q'',Q''',\ldots\), where
    $i=0,1,2,\ldots$, under iteration of \(G\), is depicted. Unless
    even and odd subsequences \(Q^{(2i)}\) and \(Q^{(2i+1)}\), which
    must converge, both converge to the same point \(B\), the
    invariant curve \(C\) must oscillate between two points lying on
    the straight line \(RS\).}
\end{figure}

This section is mainly expository, as the proof can be extracted from
our {\em Inventiones} article \cite{Rychlik97}. The technique is
covered in the monograph of Hirsch, Pugh and Shub
\cite{HPS77}. Therefore, we walk the reader through the constructions
and provide some motivations leading up to the theorem, which we
formulate at the end.

Throughout this section we use the following set
\[V = \left\{(y_0,c)\in\reals^2\,:\,0<\left|c-\frac{1}{2}\right|<\frac{1}{2},\;|y_0|<\min(c,1-c)\right\}.\]
We shall consider a family of curves in \(\reals^3\),
\(\{\Gamma(y_0,c)\}\), locally represented as a graph \(y=F(x,y_0,c)\)
and passing through point \((0,y_0)\), i.e.
\[y_0=F_c(x,y_0).\]
When \(c\) is fixed, we use the notation \(F_c(x,y_0)\) instead of \(F(x,y_0,c)\)
and \(\Gamma_c(y_0)\) for the curve in \(\reals^2\) given by the equation \(y=F_c(x,y_0)\).
This emphasizes \(c\)'s role as a parameter.

We allow \((y_0,c)\in V\); equivalently, when \(c\) is fixed, we
require that \(|y_0|<\min(c,1-c)\). Moreover, we require the local
invariance condition: the curve \(\Gamma_c(y_0)\) is mapped by \(G_c\)
to another curve of the family. Which one? It is easy to verify that
\[ G_c(0,y) = (0,-y). \]
Thus, we know that the point \((0,-y_0)\) lies in the image of
\(\Gamma_c(c)\) and thus
\[G_c\left(\Gamma_c(y_0)\right) = \Gamma_c(-y_0).\]
locally in a neighborhood of \(x=0\). 

It can be seen that this invariance condition is equivalent to a functional
equation that the function \(F(x,y_0,c)\) must satisfy at least in a neighborhood
of the line \(x=0\):
\begin{equation}
  \label{equation:modified-functional}
  F\left(\frac{x}{\sqrt{x^2+(c-F(x,y_0,c))^2}} - x,-y_0,c\right)
  =\frac{-(c-F(x,y_0,c))}{\sqrt{x^2+(c-F(x,y_0,c))^2}} + (1-F(x,y_0,c)).
\end{equation}
Conceptually, both \(y_0\) and \(c\) are parameters in \(F\) and once we
fix them, we also consider the function \(F_{y_0,c}:J\to\reals\),
\(J\subseteq\reals\), given by:
\[ F_{y_0,c}(x) = F(x,y_0,c).\]
Also, conceptually, \(F\) is a 2-parameter family of ordinary
1-variable functions. We will use the resulting three notations for
\(F\) interchangeably. The domain \(J\) of \(F_{y_0,c}\) shall depend on the
parameters.  We will assume that \(J\) is an open interval, symmetric
about \(x=0\), and that its length varies continuously:
\[ J=J_\varphi(y_0,c) = \left\{x\in\reals\,:\,-\varphi(y_0,c) < x < \varphi(y_0,c) \right\}\]
where \(\varphi: V\to\reals \) is a certain positive, continuous function. Thus,
equivalently \(F:\mathcal{J}_\varphi\to\reals\) where \(\mathcal{J}_\varphi\subset\reals^3\)
is given by:
\[ \mathcal{J}_\varphi = \left\{(x,y_0,c)\in\reals^3\,:\, 0<\left|c-\frac{1}{2}\right|<\frac{1}{2},\;
|y_0|<\min(c,1-c),\; |x|<\varphi(y_0,c)\right\}. \]
For fixed \(c\), we will also consider the section of \(\mathcal{J}_\varphi\):
\[ \mathcal{J}_\varphi(c) = \left\{(x,y_0)\in\reals^2\,:\, |y_0|<\min(c,1-c),\;|x|<\varphi(y_0,c)\right\}. \]
Let
\[ H_c = \begin{cases}
  G_c     & \text{if \(c>1/2\),}\\
  G_c^{-1} & \text{if \(c<1/2\).}
\end{cases}\]
We recall that \(G_c^{-1}=G_{1-c}\). The reason for this definition is
that the theorem below is applicable either to \(G_c\) or \(G_c^{-1}\)
depending on whether \(c>1/2\) or \(c<1/2\). To maintain symmetry
of these two cases, we formulate the theorem for \(H_c\) instead of \(G_c\).
We recall that \(G_c^{-1}=G_{1-c}\) which means that it would
be sufficient to consider only the range \(c>1/2\).
\begin{theorem}
  \label{theorem:main}
  There exists a continuous, positive function
  \(\varphi:V\to\reals\) such that
  \begin{enumerate}
  \item The set \(\mathcal{J}_\varphi(c)\) is contained in the domain of \(H_c\).
  \item We have \(H_c\left(\mathcal{J}_\varphi(c)\right)\subseteq \mathcal{J}_\varphi(c)\).
  \item For every \((x,y)\in \mathcal{J}_\varphi(c)\) the limit
    \(\lim_{n\to\infty}H_c^n(x,y)\) exists and it lies in the set
    \[V_c=\{(x,y)\,:\, x=0, |y|<\min(c,1-c)\}.\]
  \item The mapping \(\pi_c:\mathcal{J}_\phi(c)\to V_c \) defined by
    \[\pi_c(x,y) = \lim_{n\to\infty}H_c^n(x,y)\]
    is real-analytic and it is a {\em fiber bundle}. In
    particular \(\pi_c(0,y)=(0,y)\); equivalently
    \(\pi_c|V_c = id_{V_c}\).
  \item The sets \(\Gamma_c(y_0)\) defined for all \(y_0\)
    s.t. \(|y_0|<\min(c,1-c)\) by the formula:
    \[\Gamma_c(y_0) = \pi_c^{-1}(0,y_0)=
    \left\{(x,y)\in \mathcal{J}_\varphi(c)\,:\, \lim_{n\to\infty}H_c^n(x,y)=(0,y_0)\right\} \]
    may also be equivalently defined by the equation
    \[ y = F(x,y_0,c) \]
    where \(F:\mathcal{J}_\varphi\to\reals\) is a unique, real-analytic function.
  \item For all \(y_0\) s.t. \(|y_0|<\min(c,1-c)\) we have
    \[H_c\left(\Gamma_c(y_0)\right)\subset \Gamma_c(-y_0).\]
    In particular,
    \[H_c^2\left(\Gamma_c(y_0)\right)\subset \Gamma_c(y_0)\]
    i.e. the curve \(\Gamma_c(y_0)\) is invariant under the action of
    \(H_c^2\). The curve \(\Gamma_c(0)\) is invariant under
    \(H_c\).
  \item We have the following
    decomposition:
    \[\mathcal{J}_\varphi(c) = \bigcup_{y_0\,:\,|y_0|<\min(c,1-c)} \Gamma_c(y_0).\]
  \item The convergence in the definition of \(\Gamma_c(y_0)\) is
    exponentially fast.  More precisely, there exist continuous
    functions \(\mu, D:V\to\reals\) such that \(0<\mu < 1\) and
    \(0<D<\infty\) and such that for all
    \((x,y)\in\mathcal{J}_\varphi(c)\) and all nonnegative integers
    \(n\):
    \[\|H_c^n(x,y)-(0,y_0)\|\leq D(y_0,c)\,\mu(y_0,c)^n.\]
  \item Each curve \(\Gamma_c(y_0)\) is tangent to the \(x\)-axis, i.e.
    it is normal to the line \(RS\). Alternatively, for every \((y_0,c)\in V\):
    \[\frac{\partial F(x,y_0,c)}{\partial x}\bigg|_{x=0} = 0. \]
  \end{enumerate}
\end{theorem}
An outline of the proof is a subject of another section. Let us note
that the curve \(\Gamma(0,c)\) and the corresponding equation
\(y=F(x,0,c)\) solve Helfenstein's functional
equation~\ref{equation:functional}.
\section{Consequences of local existence and uniqueness}
We shall start with perhaps the least understood consequence
of Theorem~\ref{theorem:main}: the non-existence of non-differentiable,
continuous 2e-curves. We will show only {\em local differentiability}
here, near points \(A\) and \(B\). 
\begin{proposition}
  \label{proposition:structure}
  Let \(c\in(0,1/2)\cup(1/2,c)\).  Let \(K\subseteq \reals^2\) be a
  closed subset contained in \(U_c\), the domain of \(G_c\), such that
  \(G_c(K) = K\). Moreover, let \(S\notin K\).  Let \(\varphi\),
  \(\mathcal{J}_\varphi(c)\) and \(\pi_c\) be given by
  Theorem~\ref{theorem:main}. Then
  \[\mathcal{J}_\varphi(c)\cap K \subseteq \bigcup_{E\in V\cap K}\pi_c^{-1}(E).\]
\end{proposition}
\begin{proof}
  The assumption \(S\notin K\) implies that \(K\) is a compact subset
  of \(U_c\). Indeed, otherwise it would either have \(S\) in it,
  or contain a sequence \(Q_n\in K\) converging to the circular portion of the boundary of \(U_c\).
  Invariance implies that \(G(Q_n)\) converges to \(S\), so \(S\in K\), 
  which contradicts our assumptions.

  To prove the main claim, let us consider a point
  \(Q\in\mathcal{J}_\varphi(c)\cap K \).  Let us consider the limit
  \(E=\lim_{n\to\infty}H_c^n(Q)\) which exists by
  Theorem~\ref{theorem:main}. The set \(K\) is closed and invariant
  under \(H_c\), and thus \(E\in K\).  Hence, \(Q\in\pi_c^{-1}(E)\).
\end{proof}
Proposition~\ref{proposition:structure} can be applied to a
hypothetical 2e-curve \(C\) of \(G_c\). Since \(C\) must intersects the line
\(RS\) at exactly two points \(A\) and \(B\) then \(K\) must be contained in the
union of two analytic Jordan arcs \(\pi^{-1}(\{A,B\})\). So, \(C\) is automatically
\emph{locally analytic} near points \(A\) and \(B\).

Proposition~\ref{proposition:structure} may also be used to deduce an
interesting property of the real (not-hypothetical!)  curves depicted
schematically in Figure~\ref{figure:homoclinic}.  It depicts two
\(G_c\)-invariant curves which are reflections of each other and
widely oscillate near one of the points \(A\) or \(B\). Let \(C\) be
one of these curves. The curve is an {\em immersed} copy of
\(\reals\), but it is not a submanifold of \(\reals^2\) and is not
closed. Nevertheless, we may apply
Proposition~\ref{proposition:structure} to the closure of \(C\). We
obtain the following result:
\begin{proposition}
  If \(C\) is an invariant set of \(G_c\) which is a submanifold
  disjoint from the set \(x=0\), and let \(\overline{C}\) be its closure.  Then either
  \(\overline{C}\) is an analytic curve or contains an
  interval which is a subset of the line \(x=0\).
\end{proposition}
The interpretation of this result is that either the curve can be
completed to an analytic submanifold of the plane by adding to it the
points of the line \(x=0\) which lie in its closure, or it must widely
oscillate near the line \(x=0\). This fact is interesting because there
have been many attempts at constructing a 2e-invariant curve by
starting with a point \(Q\) and its image \(Q'=G(Q)\), and then
connecting them somehow by a path. One can obtain an invariant curve
\(C\) by iterating this path. Authors of these attempts often make a
claim of continuity of their curve as it approaches the line
\(RS\). The above proposition shows that such attempts must fail to
construct a continuous solution, unless the constructed curve is
analytic near \(A\) and \(B\).  As we will point out, the
\emph{Inventiones} article~\cite{Rychlik97} proves that if \(C\) is a
2e-curve which is locally analytic near the line \(RS\) then it is
globally analytic. The article disproves the existence of analytic
2e-curve, and thus locally analytic also, and furthermore by above
proposition, continuous 2e-curves.

To prove that local differentiability of \(C\) near \(A\) and \(B\)
implies global differentiability at all points, a global proof is
required. Such a proof is given Theorems~3~and~4 in the {\em
  Inventiones} article \cite{Rychlik97}. We refer the reader to the
proofs there.  Let us just comment that to prove global
differentiability of \(C\) we need two facts. The first fact that the
map \(G_c^{\pm 1} \) is defined on and differentiable on the entire
hypothetical 2e-curve \(C\).  The second fact is that for every point
\(Q\in C\) the limit \(\lim_{n\to\infty}G^{\pm n}(Q)\) exists and is
either \(A\) or \(B\). Then the local differentiable structure near
\(A\) or \(B\) may be ``transplanted'' by a suitable iteration of
a local diffeomorphism \(G\) to any point of \(C\).

Theorem~\ref{theorem:main} and Proposition~\ref{proposition:structure}
imply that for every \(c\in(0,1)\), \(c\neq 1/2\), the function
\(f(x)=F(x,0,c)\) is the unique {\em continuous} solution of
Helfenstein's functional equation~\ref{equation:functional}, and this
solution is automatically analytic. A simple consequence of this
uniqueness is that every solution is automatically even:
\(f(x)=f(-x)\). Indeed, if \(f(x)\) solves the functional equation
then so does \(\tilde{f}(x)=f(-x)\), by inspection. Hence, in view of
uniqueness, \(\tilde{f}(x)\equiv f(x)\).

To gain a little bit more clarity, we make the following definition:
\begin{definition}
  We will call a pair of continuous Jordan arcs \((C_A,C_B\)) \emph{a
    local equichordal configuration with respect to the points \(R\)
    and \(S\)} iff:
  \begin{enumerate}
  \item The intersection of any straight line parallel to \(RS\) with
    \(C_A\) or \(C_B\) consists of at most one point.
  \item The intersection of \(C_A\) and \(C_B\) with the straight line \(RS\)
    consists of exactly one point, denoted \(A\) and \(B\) respectively.
  \item For every pair of points \(P\) and \(Q\) such that:
    \begin{enumerate}
    \item \(P\) is in \(C_A\) and \(Q\) is in \(C_B\);
    \item the points \(P\) and \(Q\) and one of the points \(R\) or
      \(S\), are collinear;
    \end{enumerate}
    the distance \(|PQ|=1\).
  \item The arcs \(C_A\) and \(C_B\) are symmetric with respect to
    the reflection in the point \(O\), the center of the segment \(RS\),
    and with respect to the axis \(RS\).
  \end{enumerate}
\end{definition}
Less formally, if we only look at the chords of \(C\) whose ends are
close to \(A\) and \(B\), points \(R\) and \(S\) apppear equichordal
for the curve \(C\). We request that a local equichordal configuration
have the symmetry properties, which would follow if a 2e-curve existed.
However, we do not require that a local equichordal configuration be
a part of a 2e-curve.

It is clear to a reader of Helfenstein's paper that he attempts to
prove non-existence of a local equichordal configuration and deduce
from it that a 2e-curve does not exist. The twist is that there exists
a local equichordal configuration, but still no 2e-curve exists.
To Helfenstein's credit, in 1956 figures like
Figure~\ref{figure:homoclinic} were uncommon, used primarily as
counterexamples in topology. The classical example is the curve
\[ y=\sin\frac{1}{x} \]
which, together with the \(y\) axis, joined somehow together to make
the figure connected, form an example of a set that is not locally
connected. Later on, the science of chaos made a discovery that curves
like this are common in studying differential equations describing
real mechanical systems. A search on the terms \emph{homoclinic
  connection} and \emph{heteroclinic connection} yield many references
to such systems. It should be noted that first examples of this kind
were known already to Poincar\'e and thus available to Helfenstein.

The reader of Helfenstein's paper can easily verify that Helfenstein
tries to prove the non-existence of \(f(x)\) described by
Theorem~\ref{theorem:main}.  Since Theorem~\ref{theorem:main} shows
that the localized version of the equichordal point problem has a
solution for all \(c\) in \(0,1\), \(c\neq 1/2\), this means that any
local non-existence argument focused on a small neighborhood of the
line \(RS\) must fail.  This is a deeper reason why the problem
had remained open for 80 years until our article \cite{Rychlik97}.
\begin{corollary}
  For every \(c\) in the range \(0<c<1\) Helfenstein's functional
  equation has a unique solution \(f(x)\) defined in a neighborhood of
  \(x=0\), and which is infinitely differentiable and for all
  \(n\). Moreover, the solution for \(c\) and \(1-c\) is the same, which
  implies that for all \(c\):
  \[ a_n(c) \equiv a_n(1-c). \]
\end{corollary}
This corollary is the consequence of the symmetry with respect to
the straight line passing through \(O\) and perpendicular to the
line \(AB\). This symmetry exchanges the role of \(c\) and
\(1-c\). Thus, if a function \(f(x)\) solves the problem for \(c\)
then it also solves it for \(c\) replaced with \(1-c\).
Helfenstein was clearly aware of this corollary, but made
an incorrect assumption about the existence and uniqueness.
\section{Equichordal point problem solved, after all}
Although \(a_n(c)\equiv a_n(1-c)\) for all \(n\), the absence of
2e-curves is proved by different means in the \emph{Inventiones}
article \cite{Rychlik97}.  Thus, the line of reasoning used by
Helfenstein proved to be one of the numerous traps that await anyone
studing the problem. The Helfenstein's invariance condition appears to
have no significance in the \emph{Inventiones} solution.

As we have shown, the function \(f(x)\) with all the symmetry
properties, which is also a solution to the Helfenstein's
functional equation, does exist \emph{locally} in a neighborhood
of \(x=0\). It is the locality of the argument that is the main
error in Helfenstein's paper.  Only a \emph{global argument},
which takes into account the behavior of \(f(x)\) to the point of
breakdown of its properties as a \emph{single-valued} function,
can eliminate the possibility of a solution to the equichordal
point problem.

Helfenstein's argument is purely local, i.e. it does not refer to
the behavior of the function outside a small neighborhood of
zero. It is clear that the entire oval cannot be represented as a
graph of a single-valued function \(f(x)\).  For instance, the
standard unit circle is often represented as the graph of the
equation \(y=\pm \sqrt{1-x^2}\), where \(|x|<1\). However, the
right-hand side has \emph{two} possible values. Obviously, any
convex oval can be represented by a two-valued function with two
"branch points", when the two values coalesce to form a closed
curve.

Thus, it is clear that the breakdown of the representation of \(C\) as
a graph of an equation \(y=f(x)\) must occur at some point, something
that Helfenstein does not discuss. Other papers deal with this
issue. Wirsing in his 1958 article \cite{Wirsing58}, and Sh\"afke and
Volkmer in their 1992 article \cite{ShafkeVolkmer92}, represent the
curve in polar coordinates and use the equation \(r=g(\theta)\) to
represent the curve. This equation does not suffer from the limit on
the range, and is capable of capturing the solution to the
Helfenstein's equation near \(A\) and \(B\)
\emph{simultaneously}. Moreover, \(g(\theta)\equiv 1/2\) represents
the unit circle, which naturally plays a special role in the
asymptotic considerations as \(c\to 1/2\).  The Wirsing and Sh\"afke
and Volkmer research reveals the nature of the obstruction to the
existence of an oval with two equichordal points: if a solution is
well-behaved near \(A\), it must lose continuity near \(B\) and vice
versa. Hence, there is no \emph{globally defined} solution in polar
coordinates, either.
\section{An outline of the proof of Theorem~\ref{theorem:main}}
The first step of the proof it to interpret the solution of the
functional equation~\ref{equation:modified-functional} together with
its side condition:
\begin{eqnarray*}
  F\left(\frac{x}{\sqrt{x^2+(c-F(x,y_0,c))^2}} - x,-y_0,c\right)
  &=&\frac{-(c-F(x,y_0,c))}{\sqrt{x^2+(c-F(x,y_0,c))^2}} + (1-F(x,y_0,c)), \\
  y_0&=&F(x,y_0,c).
\end{eqnarray*}
as a question about invariant manifolds. We define
a map \(\mathcal{G}:U \to \reals^3\), where \(U \subset \reals^3\)
has already been defined before:
\[ U = \left\{(x,y,c)\in\reals^3\,:\, 0<c<1,\;0<x^2+(y-c)^2<1 \right\}. \]
defined by the formula \(\mathcal{G}(x,y,c) = (\xi,\eta,c) \) where
\begin{eqnarray*}
  \xi  &=& \frac{x}{\sqrt{x^2+(c-y)^2}} - x \\
  \eta &=& \frac{-(c-y)}{\sqrt{x^2+(c-y)^2}} + (1-y).
\end{eqnarray*}
Also see equations~\ref{equation:helfenstein-xi}-\ref{equation:helfenstein-eta}.
It is also true that \(\mathcal{G}(x,y,c)=(G_c(x,y),c)\), using
our prior notation, i.e. it simply extends \(G_c\) to three dimensions
by adding a trivial action on the parameter \(c\).

It can be seen that the surface \(W\) given by the equation
\(y=F(x,y_0,c)\) is locally invariant under the mapping
\(\mathcal{G}\) iff \(F(x,y_0,c)\) solves its functional equation,
i.e. \(\mathcal{G}(W\cap U)\subseteq W\). Because \(|y|<\min(c,1-c)\), we do not
have to worry about issues such as non-differentiability of
\(\mathcal{G}(x,y,c)\) or \(\mathcal{G}^{-1}(x,y,c)\) when \(x^2+(c-y)^2 = 0\), i.e. at
the point \((0,c,c)\); it is outside of \(V\). Moreover, every point
of \(V\) is a periodic point of \(\mathcal{G}\) of period 2, because it is easy
to see that:
\[ \mathcal{G}(0,y,c) = (0,-y,c). \]
Thus, the point \((0,0,c)\) (corresponding to the point \(A\)
in Helfenstein's paper) is a fixed point of \(\mathcal{G}\). 

In particular, \(V\) is an \emph{invariant manifold} of dimension
2. It is a union of two open segments of a straight line. The next
step in the proof is to notice that the manifold \(V\) is
\emph{normally hyperbolic} in the sense of Hirsch, Pugh and Shub
\cite{HPS77}. Without repeating lengthy definitions, we will explain
what this means. We start with linearizing \(\mathcal{G}\) at all points of
\(V\):
\[ \mathcal{G}(u,y+v,c+w) = \mathcal{G}(0,y,c) + D\mathcal{G}(0,y,c)\cdot (u,v,w) + O(\|(u,v,w)\|^2).\] 
where \(D\mathcal{G}(0,y,c)\) is the (Fr\'echet) derivative of \(\mathcal{G}\), i.e. a
\(2\times 2\) matrix. With a little bit of work, or using Maxima
program presented in Appendix~B, we find:
\[ D\mathcal{G}(0,y,c) = \left[\begin{array}{ccc}
  -\frac{1}{y-c}-1 & 0 & 0\\
  0               &-1 &0\\
  0               &0  &1
\end{array}\right]
 \]
 It happens that \(D\mathcal{G}(0,y,c)\) is diagonal, and thus it has three
 eigenvectors: \(\mathbf{e}_1=(1,0,0)\), \(\mathbf{e_2}=(0,1,0)\) and
 \(\mathbf{e_3}=(0,0,1)\), with eigenvalues:
\begin{eqnarray*}
  \lambda_1 &=& \frac{1}{c-y}-1,\\
  \lambda_2 &=& -1,\\
  \lambda_3 &=& 1,\\
\end{eqnarray*}
respectively.

What is important is that the tangent space to \(V\), spanned by
the set \(\{\mathbf{e}_2,\mathbf{e}_3\}\), is invariant and
\(|\lambda_2|=|\lambda_3|=1\). This means that under iteration of
\(\mathcal{G}\) the manifold \(V\) does not expand or contract. It
is also important that in the normal direction to \(V\) we have
linear contraction or expansion. This can be verified by
performing two iterations of \(\mathcal{G}\), starting from the point
\((0,y,c)\). In two iterations, the direction \(\mathbf{e}_1\)
scales by the product of the first eigenvalues for
\(D\mathcal{G}(0,y,c)\) and \(D\mathcal{G}(0,-y,c)\), i.e. the multiplier
\[\mu = \left(\frac{1}{c-y}-1\right)\cdot \left(\frac{1}{c+y}-1\right) =
\frac{(1-c)^2-y^2}{c^2-y^2} \]
In particular, for \(|y|<\sqrt{\min(c,1-c)}\) the multiplier
\(\mu\) is positive. If \(0< c < 1/2\), the numerator is
larger than the denominator and \(\mu < 1\). If \(1/2< c < 1\)
then the numerator is smaller than the denominator, and
\(\mu>1\). Of course, when \(c=1/2\) then \(\mu=1\).

Let us consider the partition \(V=V^+ \cup V^-\), where
\begin{eqnarray*}
  V^+ &= V\cap\left\{(x,y,c)\,:\,c>\frac{1}{2}\right\}, \\
  V^- &= V\cap\left\{(x,y,c)\,:\,c<\frac{1}{2}\right\}.
\end{eqnarray*}
The manifolds  \(V_+\) and \(V_-\) are both invariant and
satisfy the definition of \emph{normal hyperbolicity},
i.e. under the linearization the normal direction (which
in our case is \(\mathbf{e}_1\)) is scaled by a multiplier
\(\mu \neq 1\). The multiplier \(\mu\) may depend
on the point of \(V_\pm\) (and it does!), but it 
may not approach \(1\) along the trajectory of a point,
i.e. along the set
\[ \{(0,y,c),\,(\mathcal{G}(0,y,c),\,\mathcal{G}(\mathcal{G}(0,y,c)),\,\mathcal{G}(\mathcal{G}(\mathcal{G}(0,y,c))),\,\ldots\}.\]
Since in our case this set consists of two points:
\[ \{(0,y,c),\,(0,-y,c)) \}.\]
Thus, in our case of normal hyperbolicity is easy to verify.

The theory of normally hyperbolic invariant manifolds
immediately implies Theorem~\ref{theorem:main}. In passing, we verified
all technical assumptions necessary to apply it. Stating
the detailed results would involve a large amount of notations
and definitions, so we shall not repeat it here, and refer
an interested reader to the classical presentation \cite{HPS77}.

We note that our setup satisfied the definition of "immediate absolute
$r$-normal hypebolicity" and thus it satisfies the strongest
assumptions used in the classical monograph of Hirsch, Pugh and Shub
\cite{HPS77}.  Thus, the strongest conclusions also hold. It takes
very superficial understanding to conclude the existence of a
continuous \(F(x,c)\), infinitely differentiable over \(x\), with
derivatives continuous in \(c\). This suffices to prove \(a_n(c)\equiv
a_n(1-c)\). With better understanding, one can show joint analyticity
in \(x\) and \(c\).

It should also be noted that our presentation reflects the modern view
point, but Wirsing \cite{Wirsing58} proved a theorem in {\em
  Abschnitt~4} similar to Theorem~\ref{theorem:main}. He covers only
the analytic case, but the proof is very concise and to the point, and
even today it has some appeal despite the lack of generality.
\section{Conclusions}
By repeating the calculations in Helfenstein's paper
\cite{Helfenstein56} we identified the mistake in the paper. The
existence of the mistake was asserted by Wirsing \cite{Wirsing58}, but
he did not directly identify the place in which it occurred, which led
to half a century of confusion in regard to the validity of
Helfenstein's paper. We resolved this issue to our satisfaction. In
addition, we showed that Helfenstein's method cannot be a basis of a
new, simpler solution of the equichordal point problem than the one
currently known \cite{Rychlik97}.
\section{Notes and references}
\subsection{Theorem~\ref{theorem:main}}
Theorem~\ref{theorem:main} corresponds to Theorem~2 of the
\emph{Inventiones} article \cite{Rychlik97}. There is a difference of
notation, because in these notes we adopted Helfenstein's notation.
The symmetry properties are introduced in the text of the
\emph{Inventiones} article in the parts leading up to Theorem~2.

It should be noted that Theorem~\ref{theorem:main} is in the "easy"
part of the article, and therefore any determined reader is capable of
understanding it with minimal effort, given some familiarity with
invariant manifold theory. Today, invariant manifold theory is quite
well understood, with many excellent expositions. We still prefer the
work of Hirsch, Pugh and Shub \cite{HPS77}.

The analyticity of \(F(x,c)\) follows from an argument of
\cite{Rychlik97} used to prove Theorem~7 there.  However, the argument
is not easy to separate from a more complex situation it is
addressing. Finally, the analytic version of
Theorem~\ref{theorem:main} follows from {\em Abschnitt~4} of Wirsing's
1958 article \cite{Wirsing58}. Wirsing was not the first one to invent
the method of proof. It goes back to the works of Hadamard and Perron,
and this fact is well known today. The method is used by Hirsch, Pugh
and Shub \cite{HPS77}, and the references to Hadamard and Perron can
be found there.
\subsection{Comments by B. Gr\"unbaum}
The following comments can be found in the notes by B. Gr\"unbaum
dated 2010 included as part of not yet published revision of the book
by V.~Klee \cite{KleeRevision60}:
\begin{quote}
  But more unexpected has been the reaction to the work of Helfenstein
  [Hel56]. His rather infelicitously formulated claim is: ``We shall
  show in this paper the non-existence of real and, in one special
  point, at least six times differentiable 2e-curves.'' What his
  argumentation shows (assuming there are no errors, and nobody
  pointed out any specific errors in the paper) is that a
  contradiction is reached from the assumption of existence of a
  2e-curve together with the assumption that the curve assumed to
  exist is six times differentiable at a specific point. This would
  seem a reasonable result, and the non-existence would be established
  provided the differentiability assumption could be proved for all
  2e-curves assumed to exist. Hence one would think (and this was
  explicitly stated by S\"uss [S\"us]) that the non-existence of u u
  2e-curves would follow from Helfenstein's result together with the
  following result of Wirsing [Wir58]: If a 2e-curve exists, it must
  be analytic, that is, infinitely differentiable at all
  points. Indeed, to belabor the completely obvious, if a curve is
  infinitely differentiable at all points, then it is six times
  differentiable at a specific point --- and the contradiction reached
  by Helfenstein proves the non-existence. It is mystery to me why
  Wirsing thought that Helfenstein's work must be in error since ``it
  is contradicted by the present investigation'' (``durch die
  vorliegende Untersuchung widerlegt wird''). But an even greater
  mystery is why Klee, in all his discussions of equichordality,
  accepted Wirsing's statement as valid, and Helfenstein's result as
  invalid. Wirsing committed a logical error, and Klee and others
  uncritically accepted it a valid. The reviews of Helfenstein's paper
  in the Math. Reviews and the Zentralblatt fail to claim any errors
  in it. Not all details of Helfenstein's proof are given in [Hel56] -
  but no error has been found either.
\end{quote}
Of course, [Hel56] refers to the earlier cited article \cite{Helfenstein56}.
The Wirsing's work cited as [Wir58] is \cite{Wirsing58}.
The last citation in the above comment, [S\"uss], is \cite{Suss}.

Clearly, this inaccurate account of the state of the equichordal point
problem still lingers in the public domain.  Gr\"unbaum writes in the
introduction (\cite{KleeRevision60}, p.~i):
\begin{quote}
  About mid-May 2010 it occurred to me that it might be appropriate to
  have the fifty-years old collection made available to participants
  at the ``100 Years in Seattle'' conference.
\end{quote}

In contrast with Gr\"unbaum, we are certain that Wirsing rejected
Helfenstein's argument for the right reasons: it contradicted
Wirsing's own research. Wirsing refers to Helfenstein's work in two
places.  The main point is made in his \emph{Abschnitt~4}, in which
Wirsing proves a variant of our Theorem~\ref{theorem:main}, and at the
end he writes:
\begin{quote}
  Es bleibt die Frage, ob f\"ur irgendwelche $c$ und $X$ die
  analytische Fortsetzung von $\frak{C}_1$ and $\frak{C}_2$ zu
  geschlossen Kurven f\"uhren kann.

  \"Ubrigens steht die Tatsache, da{\ss} die Doppelspeichenbedingung
  sich wenigstens lokal (in Umgebungen von $T_1$ and $T_2$) durch
  regul\"ar analytische Kurven befriedigen l\"a{\ss}t, im Gegensatz zu
  der Arbeit von {\sc Helfenstein} [5], der durch 6-malige
  Differentiation der Funktionalgleichung bei $T_1$, $T_2$ zu einem
  Widerspruch gelangt und daraus auf die Nichtenxistenz einer 6-mal
  differenzierbaren D-Kurve schlie{\ss}t. Der Fehler d\"urfte in den
  unver\"offentlichten Rechnungen liegen.
\end{quote}
\section{Asymptotic analysis beyond all orders}
Sh\"afke and Volkmer \cite{ShafkeVolkmer92} conducted a deep
asymptotic analysis of the equichordal point problem, which implies
that a 2e-curve may exist only for a finite set of values of \(c\).
In spirit, the paper continues the researches of Wirsing. Wirsing
proved that if \(c\to 1/2\) then the hypothetical 2e-curve must be
extremely close to the circle of radius \(1/2\) centered at \(O\),
closer than \(C_\alpha\,|c-1/2|^\alpha\) where \(\alpha\) is an
arbitrarily large power, and \(C_\alpha\) is a certain constant,
depending on \(\alpha\).  Thus the name ``asymptotic analysis beyond
all orders'' which is sometimes used in reference to this kind of
result.

Sh\"afke and Volkmer quantified this statement even further,
expressing the leading asymptotics of the difference between the
2e-curve and the circle in terms of the exponential. In fact, the
paper extends the representation of the arcs \(C_A\) and \(C_B\)
in the discussion following Theorem~\ref{theorem:main} to the full angle \(\theta\)
in polar coordinates except one point (either zero or \(\pi\))
near which the arcs start rapidly oscillating, breaking
continuity!

\newpage
\appendix
\section{Maxima code showing Helfenstein's mistake}
\begin{small}
\begin{verbatim}
/* -*- Mode: Maxima -*-
 * Author: Marek Rychlik,  Date: January 20, 2012
 * This program finds the coefficients a[n] of the Helfenstein's
 * paper by expanding f(x) in a Taylor series about x=0
 * and comparing coefficients at like powers.
 * Additionally, it factors a[n], and finds b[n] obtained
 * from a[n] by substitution c = 1/2 + sqrt(z)/2.
 */
assume(c>0);
/* Define Taylor expansion of f(x) with generic coefficients a[k]; 
the meaning of a[k] is consistent with Helfenstein */
deftaylor(f(x), sum(a[k]*x^k,k,2,inf));
/* The denominator */
d(x,y):=sqrt(x^2+(c-y)^2);
ksi(x,y):=x/d(x,y)-x;
eta(x,y):=-(c-y)/d(x,y)+(1-y);
/* The functional equation */
eqn:f(ksi(x,f(x)))=eta(x,f(x));
/* The maximum derivative order we consider */
nlimit:10;
/* Expand functional equation in a Taylor series */
tay_eqn:taylor(eqn,x,0,nlimit);
/* Extract coefficients at relevant powers */
sys_eqn:makelist(coeff(tay_eqn,x,k),k,0,nlimit);
/* Solve for the coefficients a[k] */
soln:solve(sys_eqn,makelist(a[k],k,2,nlimit));
/* Assign coefficients */
for n:2 thru nlimit do (
  a[n]:ev(a[n],soln),
  a_factored[n]:factor(a[n])
);
/* Translate coefficients to TeX */
for n:2 thru nlimit step 2 do (
  tex('a[n] = a[n]),
  tex('a[n] = a_factored[n])
  );
/* Find the coefficients b[k] */
for n:2 thru nlimit step 2 do b[n]:ratsimp(ev(a[n],c=1/2+sqrt(z)/2));
/* Translate coefficients b[k] to TeX */
for n:2 thru nlimit step 2 do (
  tex('b[n] = b[n])
  );
\end{verbatim}
\end{small}
\newpage
\section{Maxima code finding Fr\'echet derivative of $G$}
\begin{small}
\begin{verbatim}
/* -*- Mode: Maxima -*-
 * Author: Marek Rychlik,  Date: January 20, 2012
 * This program finds the derivative DG(x,y,c)
 * and simplified expression for DG(0,y,c).
 */
kill(all); assume(c>0); assume(y<c);
/* The denominator */
d(x,y,c):=sqrt(x^2+(c-y)^2);
ksi(x,y,c):=x/d(x,y,c)-x;
eta(x,y,c):=-(c-y)/d(x,y,c)+(1-y);
/* The mapping G */
G(x,y,c):=[ksi(x,y,c),eta(x,y,c),c];
/* The Frechet derivative of G */
A:matrix(diff(G(x,y,c),x),diff(G(x,y,c),y),diff(G(x,y,c),c));
/* Typeset the derivative as TeX */
tex(A);
/* Compute the Frechet derivative at a point of V */
B:ev(A,[x=0]);
/* Clean up by expanding into partial fractions */
B:partfrac(B,y);
/* TeX form of the derivative at a point  of V */
tex(B);
\end{verbatim}
\end{small}
\section{Maxima code to verify formula for inverse of $G$}
\begin{small}
\begin{verbatim}
/* -*- Mode: Maxima -*-
 * Author: Marek Rychlik,  Date: February 1, 2012
 * This program provides an arithmetical proof that
 * G_c and G_{1-c} are inverses of each other.
 */
/* The denominator */
kill(all);
assume(c>0,c<1,y<c,y-c+1<0,y+c>0,y+c-1>0);
/* The denominator */
d(x,y,c):=sqrt(x^2+(c-y)^2);
u(x,y,c):=x-x/d(x,y,c);
v(x,y,c):=y+(c-y)/d(x,y,c);
ksi(x,y,c):=-u(x,y,c);
eta(x,y,c):=1-v(x,y,c);
/* The map */
G(x,y,c):=[ksi(x,y,c),eta(x,y,c)];
/* The generic point near B */
Q:[x,y];
/* Point P is collinear with S and Q, such that dist(P,Q)=1 */
/* Note: P is expressed in xy-coordinates */
P:[u(x,y,c),v(x,y,c)];
/* Reflection of Q; also expressed in xy-coordinates */
P1:[-x, 1-y];
/* Q1 = G(Q) */
Q1:[ksi(x,y,c),eta(x,y,c)];
/* Checks that dist(Q1,P1)=1 */
/* Must print "true"! */
is(ratsimp(apply("+",(Q1-P1)^2))=1);
/* Checks that P1,Q1 and R are collinear (ksi-eta coordinate system) */
R:[0,1-c];
/* Must print "true"! */
is(ratsimp(determinant(matrix(P1-R,Q1-R)))=0);
\end{verbatim}
\end{small}
\end{document}